\documentclass{article}%
\usepackage{graphicx}
\usepackage{amsmath}
\usepackage{amsfonts}
\usepackage{amssymb}%
\providecommand{\U}[1]{\protect\rule{.1in}{.1in}}
\newtheorem{theorem}{Theorem}
{}

\newtheorem{condition}{Condition}

\newtheorem{corollary}{Corollary}

\newtheorem{lemma}{Lemma}
{}

\newtheorem{remark}{Remark}

\newenvironment{proof}[1][Proof]{\textbf{#1.} }{\ \rule{0.5em}{0.5em}}

\begin{document}

\title{On\ the Bands of the Schr\"{o}dinger Operator with a Matrix Potential}
\author{O.A.Veliev\\{\small Dogus University, \ Istanbul, Turkey.}\\\ {\small e-mail: oveliev@dogus.edu.tr}}
\date{}
\maketitle

\begin{abstract}
In this article we consider the one-dimensional Schr\"{o}dinger operator
$L(Q)$ with a Hermitian periodic $m\times m$ matrix potential $Q$. We
investigate the bands and gaps of the spectrum and prove that most of the
positive real axis is overlapped by $m$ bands. Moreover, we find a condition
on the potential $Q$ for which the number of gaps in the spectrum of $L(Q)$ is finite.



\end{abstract}

\section{Introduction and Preliminary Facts}

Let $L(Q)$ be the differential operator generated in the space $L_{2}%
^{m}(-\infty,\infty)$ of the vector functions $y=\left(  y_{1},y_{2}%
,...,y_{m}\right)  $ by the differential expression%
\begin{equation}
-y^{^{\prime\prime}}+Qy, \tag{1}%
\end{equation}
where $y_{k}\in L_{2}(-\infty,\infty)$ for $k=1,2,...,m,$ $Q(x)=\left(
q_{s,j}(x)\right)  $ is a $m\times m$ Hermitian matrix for all $x\in
(-\infty,\infty)$, $q_{s,j}$ is the complex-valued locally square summable
function and $Q\left(  x+1\right)  =Q\left(  x\right)  $. It is well-known
that [2, Chap.XIII], [4, 10, 13] the spectrum $\sigma(L(Q))$ of the operator
$L(Q)$ is the union of the spectra of the operators $L_{t}(Q)$ for $t\in
(-\pi,\pi],$ where $L_{t}(Q)$ is the operator generated in $L_{2}^{m}[0,1]$ by
the differential expression (1) and the quasiperiodic conditions $y\left(
1\right)  =e^{it}y\left(  0\right)  ,$ $y^{^{\prime}}\left(  1\right)
=e^{it}y^{^{\prime}}\left(  0\right)  .$ For $t\in(-\pi,\pi]$ the spectra
$\sigma(L_{t}(Q))$ of the operators $L_{t}(Q)$ consist of the eigenvalues
\begin{equation}
\lambda_{1}(t)\leq\lambda_{2}(t)\leq\cdot\cdot\cdot\tag{2}%
\end{equation}
called the Bloch eigenvalues of $L(Q)$. The $n$-th band function $\lambda_{n}$
continuously depends on $t$ and its range
\begin{equation}
I_{n}(Q)=\left\{  \lambda_{n}(t):t\in(-\pi,\pi]\right\}  \tag{3}%
\end{equation}
is called the $n$-th band of the spectrum of $L$:
\[
\sigma(L(Q))=%
{\textstyle\bigcup\limits_{n=1}^{\infty}}
I_{n}(Q).
\]
The continuity of $\lambda_{n}$ in the case $m=1$ was proved in [12]. The
general case follows from the arguments of the perturbation theory described
in [6] and [12]. In the Remark 1 of the next section, for the independence of
this paper, we give a proof of this statement within the framework of this
paper. The bands $I_{n}$ approach infinity as $n\rightarrow\infty.$ The spaces
between the bands $I_{k}$ and $I_{k+1}$ (if exist) for $k=1,2,...,$ are called
the gaps in the spectrum of $L(Q).$

In this paper we investigate the set of the Bloch eigenvalues, bands and gaps
of $L(Q)$. For this first we consider the set of the Bloch eigenvalues of the
operators $L(O)$ and $L(C),$ where $O$ is the $m\times m$ zero matrix and
\begin{equation}
C=\int\limits_{[0,1]}Q\left(  x\right)  dx. \tag{4}%
\end{equation}

It is clear that%

\[
\varphi_{k,1,t}=\left(
\begin{array}
[c]{c}%
e_{k,t}\\
0\\
\vdots\\
0
\end{array}
\right)  ,\text{ }\varphi_{k,2,t}=\left(
\begin{array}
[c]{c}%
0\\
e_{k,t}\\
\vdots\\
0
\end{array}
\right)  ,...,\text{ }\varphi_{k,m,t}=\left(
\begin{array}
[c]{c}%
0\\
\vdots\\
0\\
e_{k,t}%
\end{array}
\right)
\]
are the eigenfunctions of the operator $L_{t}(O)$ corresponding to the
eigenvalue $\left(  2\pi k+t\right)  ^{2}$, where $e_{k,t}(x)=e^{i\left(  2\pi
k+t\right)  x}.$ If $t\neq0,\pi$ then the multiplicity of the eigenvalue
$\left(  2\pi k+t\right)  ^{2}$ is $m$ and the corresponding eigenspace is
\[
E_{k}(t)=Span\left\{  \varphi_{k,1,t},\varphi_{k,2,t},...,\varphi
_{k,m,t}\right\}  .
\]
In the cases $t=0$ and $t=\pi$ the multiplicities of the eigenvalues $\left(
2\pi k\right)  ^{2}$ for $k\in\left(  \mathbb{Z}\backslash\left\{  0\right\}
\right)  $ and $\left(  2\pi k+\pi\right)  ^{2}$ for $k\in\mathbb{Z}$ are $2m$
and the corresponding eigenspaces are%

\[
E_{k}(0)=Span\left\{  \varphi_{n,j,0}:n=k,-k;\text{ }j=1,2,...m\right\}
\]
and%

\[
E_{k}(\pi)=Span\left\{  \varphi_{n,j,\pi}:n=k,-(k+1);\text{ }%
j=1,2,...m\right\}
\]
respectively. Thus the points $\left(  2\pi k\right)  ^{2}$ for $k\in\left(
\mathbb{Z}\backslash\left\{  0\right\}  \right)  $ and $\left(  2\pi
k+\pi\right)  ^{2}$ for $k\in\mathbb{Z}$\ are the exceptional Bloch
eigenvalues of $L(O),$ in the sense that at these points the multiplicities of
the eigenvalues are changed. It is clear that $\left(  2\pi k+t\right)  ^{2}$
is an exceptional Bloch eigenvalue of $L(O)$ if and only if
\begin{equation}
\left(  2\pi k+t\right)  ^{2}=\left(  2\pi n+t\right)  ^{2} \tag{5}%
\end{equation}
for some $n\neq k.$ Since (5) holds only in the cases $t=0,n=-k,$ $k\neq0$ and
$t=\pi,$ $n=-k-1,$ only the points $\left(  2\pi k\right)  ^{2}$ for
$k\in\left(  \mathbb{Z}\backslash\left\{  0\right\}  \right)  $ and $\left(
2\pi k+\pi\right)  ^{2}$ for $k\in\mathbb{Z}$\ are the exceptional Bloch
eigenvalues of $L(O).$

To analyze the set of the Bloch eigenvalues and the spectrum of the operator
$L(C)$ we introduce the following notations, where $C$ is defined in (4).
Denote by $\mu_{1}<\mu_{2}<...<\mu_{p}$ the distinct eigenvalues of the
Hermitian matrix $C$. If the multiplicity of $\mu_{j}$ is $m_{j},$ then
$m_{1}+m_{2}+...+m_{p}=m$. Let $u_{j,1},$ $u_{j,2},...,u_{j,m_{j}}$ be the
eigenvectors of the matrix $C$ corresponding to the eigenvalue $\mu_{j}.$ It
is not hard to see that
\[
\Phi_{k,j,s,t}(x)=u_{j,s}e^{i\left(  2\pi k+t\right)  x}%
\]
for $s=1,2,...,m_{j}$ are the eigenfunctions of $L_{t}(C)$ corresponding to
the eigenvalue
\begin{equation}
\mu_{k,j}(t)=\left(  2\pi k+t\right)  ^{2}+\mu_{j}. \tag{6}%
\end{equation}

\bigskip To consider the spectrum of $L(Q)$ we use the following result of [16].

\textit{Theorem 4}$(a)$\textit{ of [16]. All large eigenvalues of }%
$L_{t}\left(  Q\right)  $\textit{ lie in }$\varepsilon_{k}$\textit{
neighborhood }%
\[
U_{\varepsilon_{k}}(\mu_{k,j}(t)):=(\mu_{k,j}(t)-\varepsilon_{k},\mu
_{k,j}(t)+\varepsilon_{k})
\]
\textit{of the eigenvalues }$\mu_{k,j}(t)$\textit{ of }$L_{t}(C),$\textit{
where }$\varepsilon_{k}=c_{1}(\mid\frac{\ln|k|}{k}\mid+q_{k}),$\textit{ }%
\[
q_{k}=\max\left\{  \left\vert \int\limits_{[0,1]}q_{s,r}\left(  x\right)
e^{-2\pi inx}dx\right\vert :s,r=1,2,...,m;\text{ }n=\pm2k,\pm(2k+1)\right\}
,
\]
$c_{1}$ \textit{is a constant and does not depend on }$t\in(-\pi,\pi
].$\textit{ Moreover, for each large eigenvalue }$\mu_{k,j}(t)$\textit{ of
}$L_{t}(C)$\textit{ there exists an eigenvalue of }$L_{t}(Q)$\textit{ lying in
}$\varepsilon_{k}$\textit{ neighborhood of }$\mu_{k,j}(t)$\textit{.}

Now let's explain a brief outline of this paper. Using Theorem 4$(a)$ of [16],
we first prove that most of the positive real axis is overlapped by $m$ bands
and estimate the length of the gap between the bands (see Theorems 1 and 2).
Then, in order to investigate the spectrum of $L(Q)$ in detail by using the
asymptotic formulas and perturbation theory\ we consider the multiplicities of
the eigenvalues of $L_{t}(C)$ and the large exceptional points of the spectrum
of $L(C).$ We consider the operator $L(Q)$ as perturbation of $L(C)$ by $Q-C$
and prove that the perturbation $Q-C$ may generate the gaps in $\sigma\left(
L(Q)\right)  $ only at the neighborhoods of the exceptional Bloch eigenvalues
(see Theorem 3 and Corollary 2) of $L(C)$. In Theorem 4 we find a condition
(see Condition 1) on the eigenvalues of the matrix $C$ for which the number of
gaps in the spectrum of $L(Q)$ is finite. Note that in [16] we proved Theorem
4 under the assumption that the matrix $C$ has three simple eigenvalues
$\mu_{j_{1}},$ $\mu_{j_{2}}$ and $\mu_{j_{3}}$ satisfying Condition 1. These
assumption simplifies the proof of Theorem 4. In this paper we prove Theorem 4
without any conditions on the multiplicity of these eigenvalues. Finally, note
that in [8, 14, 15] we studied the non-self-adjoint operators with a periodic
matrix potential. This paper can be considered as continuation of the paper
[16], in which the self-adjoint case was investigated. The self-adjoint case,
was considered also in [1], where the main goal was to reformulate some
spectral problems for the differential operator with periodic matrix
coefficients as problems of conformal mapping theory.

Finally note that a great number of paper is devoted to the scalar case
($m=1)$ (see for example the monographs [3] and [7] and the paper [9]). In
this paper we consider the finite-zone potentials in the vectorial case.
Therefore let us only stress the significant difference between the scalar and
vectorial cases in the investigations of the finite-zone potentials. In case
$m=1$ the finite zone potentials are infinitely differentiable functions and
have a special form expressed by Riemann $\theta$ function (see [7, Chapters 8
and 9] and [5]), while in the vectorial case we guarantee finite number of
gaps under simple algebraic condition on the eigenvalue of the matrix $C$.
Moreover, the method used in this paper for the investigation of the vectorial
case is absolutely different from the methods used in the scalar case.

\section{Main Results}

One can easily verify that the set
\[
\left\{  \left(  2\pi k+t\right)  ^{2}:t\in(-\pi,\pi],\text{ }k\in
\mathbb{Z}\right\}
\]
of Bloch eigenvalues $\left(  2\pi k+t\right)  ^{2}$ of $L(O)$ overlap $2m$
times (counting the multiplicity) the half line $(0,\infty)$, since $\left(
2\pi k+t\right)  ^{2}=\left(  -2\pi k-t\right)  ^{2}$ for $t\in(0,\pi)$ and
the multiplicity of the eigenvalues $\left(  2\pi k\right)  ^{2}$ for
$k\in\left(  \mathbb{Z}\backslash\left\{  0\right\}  \right)  $ and $\left(
2\pi k+\pi\right)  ^{2}$ for $k\in\mathbb{Z}$ is $2m.$ Since the both Bloch
eigenvalues $\left(  2\pi k+t\right)  ^{2}$ and $\left(  -2\pi k-t\right)
^{2}$ belong to the same band of $L(O),$ any element of the half line
$(0,\infty)$ is overlapped by $m$ bands of $L(O)$. To investigate this
overlapping problem for the operator $L(Q)$ let us note that the eigenvalues
of $L_{t}(Q)$ numbered in non-decreasing order (see (2)) continuously depend
on $t\in(-\pi,\pi].$

\begin{remark}
Here we prove (within the framework of this paper) that the eigenvalues
$\lambda_{n}(t)$ defined in (2) continuously depend on $t$. The eigenvalues of
the operator $L_{t}(Q)$ are the roots of the characteristic determinant
\[
\Delta(\lambda,t)=\det(U_{v}(Y_{j}))_{j,\nu=1}^{2}=
\]%
\[
e^{i2mt}+f_{1}(\lambda)e^{i(2m-1)t}+f_{2}(\lambda)e^{i(2m-2)t}+...+f_{2m-1}%
(\lambda)e^{it}+1
\]
which is a polynomial of $e^{it}$\ with entire coefficients $f_{1}%
(\lambda),f_{2}(\lambda),...$, where
\[
U_{v}(Y_{j})=Y_{j}^{(\nu-1)}(1,\lambda)-e^{it}Y_{j}^{(\nu-1)}(0,\lambda),
\]
$Y_{1}(x,\lambda)$ and $Y_{2}(x,\lambda)$ are the solutions of the matrix
equation
\[
-Y^{^{\prime\prime}}(x)+Q\left(  x\right)  Y(x)=\lambda Y(x)
\]
satisfying $Y_{1}(0,\lambda)=O$, $Y_{1}^{^{\prime}}(0,\lambda)=I$ and
$Y_{2}(0,\lambda)=I$, $Y_{2}^{^{\prime}}(0,\lambda)=O$ (see [11] Chapter 3).

Now using these statements we prove that for each $n$\ the function
$\lambda_{n}$ defined in (3) is continuous at each point $t_{0}\in(-\pi,\pi].$
Since $\lambda_{n}(t_{0})\rightarrow\infty$ as $n\rightarrow\infty,$ there
exist $k\leq n$ and $p\geq n$ such that $\lambda_{k-1}(t_{0})<\lambda
_{k}(t_{0})=\lambda_{k+1}(t_{0})=...=\lambda_{p}(t_{0})<\lambda_{p+1}(t_{0})$
if $\lambda_{n}(t_{0})>\lambda_{1}(t_{0}).$ Then the boundaries of the
rectangles
\[
R_{1}=\left\{  c<x<d_{1},\text{ }\left\vert y\right\vert <1\right\}  \text{
and }R_{2}=\left\{  c<x<d_{2},\text{ }\left\vert y\right\vert <1\right\}
\]
belong to the resolvent set of the operator $L_{t_{0}}(Q),$ where
$\lambda_{k-1}(t_{0})<d_{1}<\lambda_{k}(t_{0}),$ $\lambda_{p}(t_{0}%
)<d_{2}<\lambda_{p+1}(t_{0})$ and $c$ is a number for which $\sigma
(L)\subset(c,\infty).$ It implies that $\Delta(\lambda,t_{0})\neq0$ for each
$\lambda\in\partial(R_{1}).$ Since $\Delta(\lambda,t_{0})$ is a continuous
function on the compact $\partial(R_{1}),$ there exists $a>0$ such that
$\left\vert \Delta(\lambda,t_{0})\right\vert >a$ for all $\lambda\in
\partial(R_{1}).$ Moreover, $\Delta(\lambda,t)$ is a polynomial of $e^{it}$
with entire coefficients. Therefore, there exists $\delta_{1}>0$ such that
$\left\vert \Delta(\lambda,t)\right\vert >a/2$ for all $t\in(t_{0}-\delta
_{1},t_{0}+\delta_{1})$ and $\lambda\in\partial(R_{1}).$ It means that
$\partial(R_{1})$ belong to the resolvent set of $L_{t}(Q)$ for all
$t\in(t_{0}-\delta_{1},t_{0}+\delta_{1}).$ Moreover
\[
\left(  L_{t}-\lambda I\right)  ^{-1}f(x)=\int_{0}^{1}G(x,\xi,\lambda
,t)f(\xi)d\xi,
\]
where $G(x,\xi,\lambda,t)$ is the Green's function of $L_{t}-\lambda I$
defined by formula%
\[
G(x,\xi,\lambda,t)=g(x,\xi,\lambda)-\frac{1}{\Delta(\lambda,t)}\sum
\limits_{j,v=1}^{2}Y_{j}(x,\lambda)V_{jv}(x,\lambda)U_{v}(g),
\]
(see formula (8) of [7, p.117]). Here $g$ does not depend on $t$ and $V_{jv}$
is the transpose of that mth-order matrix consisting of the cofactor of the
element $U_{v}(Y_{j})$ in the determinant $\det(U_{v}(Y_{j}))_{j,\nu=1}^{2}.$
Hence the entries of the matrices $V_{jv}(x,\lambda)$ and $U_{v}(g)$ either do
not depend on $t$ or have the form $u(1,\lambda)-e^{it}u(0,\lambda)$ and
$h(1,\xi,\lambda)-e^{it}h(0,\xi,\lambda)$ respectively, where the functions
$u$ and $h$ do not depend on $t.$ Therefore using these formulas and the last
inequality for $\left\vert \Delta(\lambda,t)\right\vert $ one can easily
verify that $\left(  L_{t}-\lambda I\right)  ^{-1}$ continuously depend on
$t\in(t_{0}-\delta_{1},t_{0}+\delta_{1})$ for $\lambda\in\partial(R_{1}).$
This implies that the operators $L_{t}$ for each $t\in(t_{0}-\delta_{1}%
,t_{0}+\delta_{1})$ have \ $k-1$ eigenvalues in $R_{1},$ since $L_{t_{0}}$
have \ $k-1$ eigenvalues in $R_{1}$. It is clear that these eigenvalues are
$\lambda_{1}(t),\lambda_{2}(t),...,\lambda_{k-1}(t).$ In the same way we prove
that there exists $\delta_{2}>0$ such that the operators $L_{t}$ for
$t\in(t_{0}-\delta_{2},t_{0}+\delta_{2})$ have \ $p$ eigenvalues in $R_{2}$
and they are $\lambda_{1}(t),\lambda_{2}(t),...,\lambda_{p}(t).$ Thus the
closed rectangle $R=\left\{  d_{1}\leq x\leq d_{2},\text{ }\left\vert
y\right\vert \leq1\right\}  $ contains $p-k+1$ eigenvalues of $L_{t}$ for
$t\in(t_{0}-\delta,t_{0}+\delta)$ and they are $\lambda_{k}(t),\lambda
_{k+1}(t),...,\lambda_{p}(t),$ where $\delta=\min\left\{  \delta_{1}%
,\delta_{2}\right\}  $ and $n\in\lbrack k,p].$

Now we are ready to prove that $\lambda_{n}$ is continuous at the point
$t_{0}$ if $\lambda_{n}(t_{0})>\lambda_{1}(t_{0}).$ Consider any sequence
$\left\{  \left(  \lambda_{n}(t_{k}),t_{k}\right)  :k\in\mathbb{N}\right\}  $
such that $t_{k}\in(t_{0}-\delta,t_{0}+\delta)$ for all $k\in\mathbb{N}$ and
$t_{k}\rightarrow t_{0}$ as $k\rightarrow\infty.$ Let $\left(  \lambda
,t_{0}\right)  $ be any limit point of the sequence $\left\{  \left(
\lambda_{n}(t_{k}),t_{k}\right)  :k\in\mathbb{N}\right\}  .$ Since $\Delta$ is
a continuous function with respect to the pair $(\lambda,t)$ and
$\Delta\left(  \lambda_{n}(t_{k}),t_{k}\right)  =0$ for all $k$ we have
$\Delta(\lambda,t_{0})=0.$ It means that $\lambda$ is an eigenvalue of
$L_{t_{0}}(Q)$ lying in the rectangle $R$, that is, $\lambda=\lambda_{k}%
(t_{0})=$ $\lambda_{k+1}(t_{0})=...=\lambda_{p}(t_{0}),$ where $n\in\lbrack
k,p].$ Thus $\lambda_{n}(t_{k})\rightarrow\lambda_{n}(t_{0})$ as
$k\rightarrow\infty$ for any sequence $\left\{  t_{k}:k\in\mathbb{N}\right\}
$ converging to $t_{0}$ and $\lambda_{n}$ is continuous at the point $t_{0}$
if $\lambda_{n}(t_{0})>\lambda_{1}(t_{0}).$ To prove the case $\lambda
_{n}(t_{0})=\lambda_{1}(t_{0})$ it is enough to consider only the rectangle
$R_{2}=\left\{  c<x<d_{2},\text{ }\left\vert y\right\vert <1\right\}  ,$ where
$\lambda_{1}(t_{0})=\lambda_{2}(t_{0})=...=\lambda_{p}(t_{0})<d_{2}%
<\lambda_{p+1}(t_{0}).$
\end{remark}

First using this Remark and Theorem 4$(a)$ of [16] we consider the overlapping
problem for $L(Q).$ For this in the following remark we explain Theorem 4$(a)$
of [16] for the family of the operators $L(C+\varepsilon(Q-C))$ for
$\varepsilon\in\lbrack0,1].$

\begin{remark}
To prove Theorem 4$(a)$ we used the formula
\begin{equation}
(\lambda_{k,j}(t)-\mu_{n,i}(t))(\Psi_{k,j,t},\Phi_{n,i,t})=((Q(x)-C)\Psi
_{k,j,t},\Phi_{n,i,t}) \tag{7}%
\end{equation}
and proved that
\begin{equation}
\left(  \Psi_{k,j,t}(x),(Q\left(  x\right)  -C)\Phi_{n,i,t}(x)\right)
=O(\frac{\ln|k|}{k})+O(b_{k}). \tag{8}%
\end{equation}
Moreover, the last estimation does not depend on $t$. If the potential $Q$ is
replaced by $C+\varepsilon(Q-C)$ then the formula (7) has the form
\begin{equation}
(\lambda_{k,j}(t)-\mu_{n,i}(t))(\Psi_{k,j,t},\Phi_{n,i,t})=(\varepsilon
(Q(x)-C)\Psi_{k,j,t},\Phi_{n,i,t}) \tag{9}%
\end{equation}
Instead of (7) using (9) and repeating the proof of (8) we get
\[
\left(  \Psi_{k,j,t}(x),\varepsilon(Q\left(  x\right)  -C)\Phi_{n,i,t}%
(x)\right)  =O(\frac{\ln|k|}{k})+O(\varepsilon q_{k}).
\]
Therefore, repeating the proof of Theorem 4$(a)$ of [16] we obtain that
\textit{all large eigenvalues of }$L_{t}(C+\varepsilon(Q-C))$\textit{ }for all
$\varepsilon\in\lbrack0,1]$ \textit{lie in }$\varepsilon_{k}=:c_{1}(\mid
\frac{\ln|k|}{k}\mid+q_{k})$\textit{ neighborhood of the eigenvalues }%
$\mu_{k,j}(t)$\textit{ for }$\left\vert k\right\vert \geq N$\textit{ and
}$j=1,2,...p,$ where the constants $N$ and $c_{1}$ do not depend on $t$ and
$\varepsilon.$
\end{remark}

Now we are ready to prove the following theorem about the overlapping problem.

\begin{theorem}
There exists a positive integer $N_{1}$ such that if $s\geq N_{1}$ then the
intervals
\[
I(s):=\left[  (s\pi)^{2}+\mu_{p}+\varepsilon(s),\left(  s\pi+\pi\right)
^{2}+\mu_{1}-\varepsilon(s)\right]
\]
are contained in each of the bands
\[
I_{sm+1}(Q),I_{sm+2}(Q),...,I_{sm+m}(Q),
\]
where $\varepsilon(s)=\varepsilon_{k}$ if $s\in\left\{  2k,2k+1\right\}
$\ and $\varepsilon_{k}$ is defined in Remark 2.
\end{theorem}

\begin{proof}
First consider the case $s=2k.$ One can easily verify that the number of the
periodic Bloch eigenvalues of $L(C)$ (the eigenvalues of $L_{0}(C)$ counting
the multiplicity) lying in the interval
\[
\left[  c,\left(  s\pi\right)  ^{2}+\mu_{p}+\varepsilon(s)\right]
\]
is $sm+m,$ where $c$ is a constant such that the spectra of the operators
$L(C+\varepsilon(Q-C))$ are contained in $(c,\infty)$ for all $\varepsilon
\in\lbrack0,1].$ It follows from Remark 2 that there exist constants $N_{1}$
and $c$ such that if $s\geq N_{1},$ then the boundary of the rectangle
\[
R_{1}=\left\{  c<x<\left(  s\pi\right)  ^{2}+\mu_{p}+\varepsilon(s),\text{
}\left\vert y\right\vert <1\right\}
\]
belong to the resolvent set of the operators $L_{0}(C+\varepsilon(Q-C))$ for
all $\varepsilon\in\lbrack0,1]$. Hence, the projection of $L_{0}%
(C+\varepsilon(Q-C))$ defined by contour integration over the boundary of
$R_{1}$ depends continuously on $\varepsilon.$ It implies that the number of
eigenvalues (counting the multiplicity) of $L_{t}(C+\varepsilon(Q-C))$ lying
in $R_{1}$ are the same for all $\varepsilon\in\lbrack0,1].$ Since $L_{0}(C)$
has $sm+m$ eigenvalues (counting the multiplicity) in $R_{1},$ the operator
$L_{0}(Q)$ has also $sm+m$ eigenvalues.

In the same way we prove that the rectangle
\[
R_{2}=\left\{  c<x<\left(  s\pi\right)  ^{2}+\mu_{1}-\varepsilon(s),\text{
}\left\vert y\right\vert <1\right\}
\]
contains $sm-m$ eigenvalues of $L_{0}(Q).$ Therefore the interval
\begin{equation}
\left(  \left(  s\pi\right)  ^{2}+\mu_{1}-\varepsilon(s),\left(  s\pi\right)
^{2}+\mu_{p}+\varepsilon(s)\right)  \tag{10}%
\end{equation}
contains $2m$ periodic eigenvalues \ and they are
\[
\lambda_{sm-m+1}(0)\leq\lambda_{sm-m+2}(0)\leq\cdot\cdot\cdot\leq
\lambda_{sm+m}(0).
\]

In the similar way we prove that the interval
\begin{equation}
\left(  \left(  s\pi+\pi\right)  ^{2}+\mu_{1}-\varepsilon(s),\left(  s\pi
+\pi\right)  ^{2}+\mu_{p}+\varepsilon(s)\right)  \tag{11}%
\end{equation}
contains $2m$ antiperiodic eigenvalues \ and they are
\[
\lambda_{sm+1}(\pi)\leq\lambda_{sm+2}(\pi)\leq\cdot\cdot\cdot\leq
\lambda_{sm+2m}(\pi)
\]
Thus the bands $I_{r}$ for $r=sm+1,sm+2,...,sm+m$ contain the point
$\lambda_{r}(0)$ from interval (10) and the point $\lambda_{r}(\pi)$ from
interval (11). Therefore the bands $I_{sm+1},I_{sm+2},...,I_{sm+m}$ contains
the interval $I(s).$ In the same way we prove the case $s=2k+1.$
\end{proof}

Theorem 1 immediately imply the following.

\begin{corollary}
Any spectral gap $(\alpha,\beta)$ with $\alpha>\left(  \pi N_{1}\right)  ^{2}$
(if exists) is contained in the intervals
\[
U(s):=(\left(  \pi s\right)  ^{2}+\mu_{1}-\varepsilon(s-1),\left(  \pi
s\right)  ^{2}+\mu_{p}+\varepsilon(s)),
\]
for $s>N_{1}$. Moreover, the spectral gap $(\alpha,\beta)\subset U(s)$ lies
between the bands $I_{sm}(Q)$ and $I_{sm+1}(Q).$
\end{corollary}

\begin{proof}
By Theorem 1 the intervals $I(s)$ for $s\geq N_{1}$ are the subsets of the
spectrum. Therefore the spectral gap $(\alpha,\beta)$ with $\alpha>\left(  \pi
N_{1}\right)  ^{2}$ is contained between $I(s-1)$ and $I(s)$ for some
$s>N_{1}.$ It means that $(\alpha,\beta)\subset U(s).$ The bands $I_{sm+j}(Q)$
and $I_{sm+j+1}(Q)$ for $j=1,2,...,m-1$ have common intervals $I(s)$ and hence
there is not gaps between they. It means that the gap $(\alpha,\beta)$ is
located between the bands $I_{ms}$ and $I_{ms+1}.$
\end{proof}

Now using Corollary 1 and Theorem 4$(a)$ of [16] we prove the following.

\begin{theorem}
The length of the spectral gaps $(\alpha,\beta)$ lying in $U(s)$ for $s>N_{1}$
is not greater than $2\max\left\{  \varepsilon(s-1),\varepsilon(s)\right\}  .$
\end{theorem}

\begin{proof}
By Corollary 1 we have
\[
(\left(  s\pi\right)  ^{2}+\mu_{1}-\varepsilon(s-1)\leq\alpha<\beta\leq\left(
\pi s\right)  ^{2}+\mu_{p}+\varepsilon(s).
\]
Now suppose on the contrary that the length of the gap $(\alpha,\beta)$ is
greater than $2\max\left\{  \varepsilon(s-1),\varepsilon(s)\right\}  .$ Then
it follows from last equalities that
\[
\frac{\alpha+\beta}{2}\in(\left(  s\pi\right)  ^{2}+\mu_{1},\left(  \pi
s\right)  ^{2}+\mu_{p}),\text{ }\frac{\beta-\alpha}{2}>\max\left\{
\varepsilon(s-1),\varepsilon(s)\right\}
\]
Using (6) one can easily conclude that there exist $t\in(-\pi,\pi]$ and
$j\in\left\{  1,2,...,p\right\}  $ such that the equality $\mu_{s,j}%
(t)=\frac{\alpha+\beta}{2}$ holds. On the other hands, by Theorem 4$(a)$ of
[16] there exists an eigenvalue $\lambda$ of $L_{t}(Q)$ lying in $\max\left\{
\varepsilon(s-1),\varepsilon(s)\right\}  $ neighborhood of $\frac{\alpha
+\beta}{2}.$ Therefore
\[
\lambda\in\left(  \frac{\alpha+\beta}{2}-\frac{\beta-\alpha}{2},\frac
{\alpha+\beta}{2}+\frac{\beta-\alpha}{2}\right)  =(\alpha,\beta).
\]
Hence $(\alpha,\beta)$ is not a gap in the spectrum of $L(Q).$ This
contradiction imply the proof of the theorem.
\end{proof}

To investigate the spectrum of $L(Q)$ in detail by using the asymptotic
formulas and perturbation theory\ we need to consider the multiplicities of
the eigenvalues of $L_{t}(C)$ and the exceptional points of the spectrum of
$L(C).$ The multiplicity of $\mu_{k,j}(t)$ is $m_{j}$ if $\mu_{k,j}(t)\neq
\mu_{n,i}(t)$ for all $(n,i)\neq(k,j).$ The multiplicity of $\mu_{k,j}(t)$ is
changed, that is, $\mu_{k,j}(t)$ is an exceptional point of the spectrum of
$L(C)$ if
\begin{equation}
\left(  2\pi k+t\right)  ^{2}+\mu_{j}=\left(  2\pi n+t\right)  ^{2}+\mu_{i}
\tag{12}%
\end{equation}
for some $(n,i)\neq(k,j).$ Since $\left(  2\pi k+t\right)  ^{2}=\left(  -2\pi
k-t\right)  ^{2}$, it is enough to study the equality (12) for $t\in
\lbrack0,\pi]$ and $k\in\mathbb{Z}.$ Moreover, we investigate the large
exceptional Bloch eigenvalues $\left(  2\pi k+t\right)  ^{2}+\mu_{j}$ of
$\sigma(L(C)),$ because we are going to investigate the spectrum of the
operator $L(Q)$ by using the asymptotic formulas for the large eigenvalues. In
other words, we need to consider (12) in the case when $\left\vert
k\right\vert $ is a large number. Then (12) has a solution $t\in\lbrack0,\pi]$
only in the cases $n=-k$ and $n=-k-1.$ In these cases (12) implies that the
large eigenvalue $\left(  2\pi k+t\right)  ^{2}+\mu_{j}$ is an exceptional
Bloch eigenvalue of $L(C)$ if either
\begin{equation}
\mu_{k,j}(t)-\mu_{-k,i}(t)=8\pi kt+\mu_{j}-\mu_{i}=0 \tag{13}%
\end{equation}
or
\begin{equation}
\mu_{k,j}(t)-\mu_{-k-1,i}(t)=4\pi(2k+1)\left(  t-\pi\right)  +\mu_{j}-\mu
_{i}=0 \tag{14}%
\end{equation}
for some $i=1,2,...,p.$ We denote by
\[
t(2k,j,i)=\frac{\mu_{i}-\mu_{j}}{4\pi(2k)}%
\]
and
\[
t(2k+1,j,i)=\pi+\frac{\mu_{i}-\mu_{j}}{4\pi(2k+1)}%
\]
the solutions of equations (13) and (14) lying in $[0,\pi]$. Thus eliminating
the sets $\ $ $\left\{  t(2k,j,i):i=1,2,...,p\right\}  $ and $\left\{
t(2k+1,j,i):i=1,2,...,p\right\}  $ from $[0,\pi]$ we conclude that, if $t$
belongs to the remaining part of $[0,\pi],$ then the multiplicity of the
eigenvalue $\mu_{k,j}(t)$ is $m_{j}$, where $k$ is a large number. In other
word, $\mu_{k,j}(t)$ is a non-exceptional Bloch eigenvalue of $L(C).$ However,
to investigate the perturbation of these non-exceptional Bloch eigenvalues by
using the asymptotic formulas obtained in [16], we eliminate $\delta_{k}%
$-neighborhoods $U_{\delta_{k}}(t(2k,j,i))$ and $U_{\delta_{k}}(t(2k+1,j,i))$
of $t(2k,j,i)$ and $t(2k+1,j,i)$ from $[0,\pi],$ where $\delta_{k}=o(k^{-1}).$
Moreover, $\delta_{k}$ can be chosen so that the remaining part of $[0,\pi]$
consists of the pairwise disjoint intervals $\left[  a(k,j,s),b(k,j,s)\right]
$ for $s=1,2,...,v:$%

\begin{equation}
\lbrack0,\pi]\backslash\left(
{\textstyle\bigcup\limits_{i=1,2,...,p}}
\left(  U_{\delta_{k}}(t(2k,j,i))\cup U_{\delta_{k}}(t(2k+1,j,i))\right)
\right)  = \tag{15}%
\end{equation}%
\[%
{\textstyle\bigcup\limits_{s=1}^{v}}
\left[  a(k,j,s),b(k,j,s)\right]  ,
\]
where $a(k,j,s)<b(k,j,s)<a(k,j,s+1)<b(k,j,s+1)$ for $s=1,2,...,v.$ These
intervals have the following property.

\begin{lemma}
There exists $N_{2}$ such that if $\mid$ $k\mid>N_{2}$,
\begin{equation}
4\pi(2\left\vert k\right\vert -2)\delta_{k}=2\max\left\{  \varepsilon
_{k},\varepsilon_{-k},\varepsilon_{-k-1}\right\}  , \tag{16}%
\end{equation}
and the quasimomentum $t$ belongs to the intervals $\left[
a(k,j,s),b(k,j,s)\right]  $ defined in (15), then the following statements hold.

$(a)$ The inequality
\begin{equation}
\left\vert \mu_{k,j}(t)-\mu_{n,i}(t)\right\vert \geq4\pi(2\left\vert
k\right\vert -1)\delta_{k} \tag{17}%
\end{equation}
holds for $n=-k,-k-1$ and for all $i=1,2,...,p.$

$(b)$ The closed interval%
\[
\overline{U_{\varepsilon_{k}}(\mu_{k,j}(t))}=\left[  \mu_{k,j}(t)-\varepsilon
_{k},\mu_{k,j}(t)+\varepsilon_{k}\right]
\]
has no common points with $\overline{U_{\varepsilon_{n}}(\mu_{n,i}(t))}\ $for
$\mid$ $n\mid\geq N_{2}$ and $(n,i)\neq\left(  k,j\right)  $.
\end{lemma}

\begin{proof}
$(a)$ Introduce the notations $f(t)=\mu_{k,j}(t)-\mu_{-k,i}(t)$ and
$g(t)=\mu_{k,j}(t)-\mu_{-k-1,i}(t).$ By the definition of $t(2k,j,i)$ and
$t(2k+1,j,i)$ we have $f(t(2k,j,i))=0$ and $g(t(2k+1,j,i))=0$ (see (13) and
(14)). On the other hand, the derivatives of the functions $f$ and $g$ are
$8\pi k$ and $4\pi(2k+1)$ respectively. Therefore if $t$ does not belong to
the $\delta_{k}$-neighborhood of $t(2k,j,i)$ and $t(2k+1,j,i)$ for
$i=1,2,...,p$ that is, if $t$ belong to the intervals $\left[
a(k,j,s),b(k,j,s)\right]  $, then $\left\vert f(t)\right\vert \geq
8\pi\left\vert k\right\vert \delta_{k}$ and $\left\vert g(t)\right\vert
\geq4\pi\left\vert 2k+1\right\vert \delta_{k}.$ These inequalities imply (17).

$(b)$ If (16) holds then it follows from (17) that the distance $\left\vert
\mu_{k,j}(t)-\mu_{n,i}(t)\right\vert $ between the centres of the intervals
$U_{\varepsilon_{k}}(\mu_{k,j}(t))$ and $U_{\varepsilon_{n}}(\mu_{n,i}(t))$ is
greater than the total sum $\varepsilon_{k}+\varepsilon_{n}$ of the radii of
these intervals for $n=-k,-k-1$. Therefore $\overline{U_{\varepsilon_{k}}%
(\mu_{k,j}(t))}$ has no common points with the intervals $\overline
{U_{\varepsilon_{n}}(\mu_{n,i}(t))}\ $for $n=-k,-k-1$. Similarly, if $n\neq
k,-k,-k-1,$ $\mid$ $k\mid>N_{2},$ $\mid$ $n\mid\geq N_{2}$ and $t\in
\lbrack0,\pi]$ then $\left\vert \mu_{k,j}(t)-\mu_{n,i}(t)\right\vert $ is a
large number and hence is greater than $\varepsilon_{k}+\varepsilon_{n}.$ If
$n=k$ and $i\neq j,$ then $\left\vert \mu_{k,j}(t)-\mu_{n,i}(t)\right\vert
=\left\vert \mu_{j}-\mu_{i}\right\vert >\varepsilon_{k}+\varepsilon_{n},$
since $\varepsilon_{k}\rightarrow0$ as $k\rightarrow\infty.$ The lemma is proved.
\end{proof}

Now using this lemma we consider the spectrum of $L(Q).$

\begin{theorem}
Let $j=1,2,...,p$ be fixed. For any interval $\left[  a,b\right]  :=\left[
a(k,j,s),b(k,j,s)\right]  $ of (15) the following statements hold.

$(a)$ If $t\in\lbrack a,b],$ then the operator $L_{t}(Q)$ has $m_{j}$
eigenvalues (counting the multiplicity) lying in the interval $U_{\varepsilon
_{k}}(\mu_{k,j}(t))=\left(  \mu_{k,j}(t)-\varepsilon_{k},\mu_{k,j}%
(t)+\varepsilon_{k}\right)  .$

$(b)$ There exists $l$ such that the eigenvalues of $L_{t}(Q)$ lying in
$U_{\varepsilon_{k}}(\mu_{k,j}(t))$ are $\lambda_{l+1}(t),\lambda
_{l+2}(t),...,\lambda_{l+m_{j}}(t)$ for all $t\in\lbrack a,b].$

$(c)$ If $k>0$ ($k<0$), then the interval $\left[  \mu_{k,j}(a)+\varepsilon
_{k},\mu_{k,j}(b)-\varepsilon_{k}\right]  $

($\left[  \mu_{k,j}(b)+\varepsilon_{k},\mu_{k,j}(a)-\varepsilon_{k}\right]  $)
is a subset of the bands $\Gamma_{l+1},$ $\Gamma_{l+2},...,\Gamma_{l+m_{j}}$
of the spectrum of $L(Q).$
\end{theorem}

\begin{proof}
$(a)$ Theorem 4$(a)$ of [16] implies that there exist $n$ and $N$ such that
the eigenvalues $\lambda_{s}(t)$ for $s>n$ lie in $U_{\varepsilon_{k}}%
(\mu_{k,i}(t))$ for $\mid$ $k\mid>N$ and $i=1,2,...,p.$ On the other hand, by
Lemma 1 the integer $N$ can be chosen so that the closed interval
$\overline{U_{\varepsilon_{k}}(\mu_{k,j}(t))}$ for $t\in\lbrack a,b]$ and
$\mid$ $k\mid>N$ has no common points with the intervals $\overline
{U_{\varepsilon_{n}}(\mu_{n,i}(t))}\ $for $\mid$ $n\mid\geq N$ and
$(n,i)\neq\left(  k,j\right)  .$ Therefore the circle
\begin{equation}
\left\{  \lambda\in\mathbb{C}:\left\vert \lambda-\mu_{k,j}(t)\right\vert
=\varepsilon_{k}\right\}  \tag{18}%
\end{equation}
belong to the resolvent set of $L_{t}(Q).$ Repeating the proof of the case
$\varepsilon=1,$ one can easily verify that the circle (18) lies in the
resolvent sets of $L_{t}(C+\varepsilon(Q-C))$ for all $\varepsilon\in
\lbrack0,1].$ Since $L_{t}(C)$ has $m_{j}$ eigenvalues (counting the
multiplicity) in the circle (18), the operator $L_{t}(Q)$ has also $m_{j}$ eigenvalues.

$(b)$ Let us denote the eigenvalues of $L_{t}(Q)$ lying in $U_{\varepsilon
_{k}}(\mu_{k,j}(t))$ by $\lambda_{l(t)+1}(t),$ $\lambda_{l(t)+2}(t),...,$
$\lambda_{l(t)+m_{j}}(t).$ We need to prove that $l(t)$ does not depend on
$t\in\lbrack a,b].$ Since we numerate the eigenvalues of $L_{t}(Q)$ in
nondecreasing order (see (2)) $\lambda_{l(t)}(t)$ and $\lambda_{l(t)+m_{j}%
+1}(t)$ do not belong to the interval $U_{\varepsilon_{k}}(\mu_{k,j}(t))$
\ and
\begin{equation}
\lambda_{l(t)}(t)<\lambda_{l(t)+s}(t)<\lambda_{l(t)+m_{j}+1}(t) \tag{19}%
\end{equation}
for all $s=1,2,...,m_{j}.$ It follows from the continuity of the band
functions and (19) that for each $t\in\lbrack a,b]$ there exists a
neighborhood $U(t)$ of $t$ such that
\[
\lambda_{l(t)}(y)<\lambda_{l(t)+s}(y)<\lambda_{l(t)+m_{j}+1}(y)
\]
for $y\in U(t).$ In the other words, $l(y)=l(t)$ for all $y\in U(t).$ Thus we
have%
\begin{equation}
\forall t\in\lbrack a,b],\exists U(t):l(y)=l(t),\forall y\in U(t). \tag{20}%
\end{equation}
Let $U(t_{1}),U(t_{2}),...,U(t_{\omega})$ be a finite subcover of the open
cover $\{U(t):t\in\lbrack a,b]\}$ of the compact $[a,b],$ where $U(t)$ is the
neighborhood of $t$ satisfying (20). By (20), we have $l(y)=l(t_{i})$ for all
$y\in U(t_{i}).$ Clearly, if $U(t_{i})\cap U(t_{j})\neq\emptyset,$ then
$l(t_{i})=l(z)=l(t_{j}),$ where $z\in U(t_{i})\cap U(t_{j})$. Thus
$l(t_{1})=l(t_{2})=...=l(t_{\omega})$ and hence $l(t)$ does not depend on
$t\in\lbrack a,b].$

$(c)$ We consider the case $k>0.$ The case $k<0$ can be considered in the same
way. Since
\[
\lambda_{l+s}(a)\in\left(  \mu_{k,j}(a)-\varepsilon_{k},\mu_{k,j}%
(a)+\varepsilon_{k}\right)
\]
and
\[
\lambda_{l+s}(b)\in\left(  \mu_{k,j}(b)-\varepsilon_{k},\mu_{k,j}%
(b)+\varepsilon_{k}\right)
\]
for $s=1,2,...,m_{j}$ the interval $\left[  \mu_{k,j}(a)+\varepsilon_{k}%
,\mu_{k,j}(b)-\varepsilon_{k}\right]  $ is a subset of $\Gamma_{l+s}$ for
$s=1,2,...,m_{j}.$
\end{proof}

Using this theorem and the construction of the intervals (15) we prove the
following consequence.

\begin{corollary}
There exists $N_{3}>\max\left\{  N,N_{1},N_{2}\right\}  $ and a sequence
$\left\{  \gamma_{k}\right\}  \rightarrow0$ such that $\gamma_{k}\rightarrow0$
as $k\rightarrow\infty$ and the spectral gap $(\alpha,\beta)$ defined in
Corollary 1 and lying in $U(k)$ for $k>N_{3}$ is contained in the intersection
of the sets $S(1,k),S(2,k),...,S(p,k),$ where
\[
S(j,k)=%
{\textstyle\bigcup\limits_{i=1,2,...,p}}
\left(  \left(  \pi k\right)  ^{2}+\frac{\mu_{i}+\mu_{j}}{2}-\gamma
_{k},\left(  \pi k\right)  ^{2}+\frac{\mu_{i}+\mu_{j}}{2}+\gamma_{k}\right)
.
\]

\end{corollary}

\begin{proof}
We say that the intervals $[A+\varepsilon,B-\varepsilon]$ and $(A-\varepsilon
,B+\varepsilon)$ are respectively the $\varepsilon>0$ contraction and
extension of the intervals $[A,B]$ and $(A,B).$ In Theorem 3 $(c)$ we proved
that the $\varepsilon_{k}$ contraction of the images $\mu_{k,j}([a,b])$ of the
intervals $[a,b]$ of (15) is a subset of the spectrum of $L(Q).$ On the other
hand, the intervals of (15) are obtained from $[0,\pi]$ by eliminating the
open intervals
\begin{equation}
(t(2k,j,i)-\delta_{k},t(2k,j,i)+\delta_{k}),\text{ }(t(2k+1,j,i)-\delta
_{k},t(2k+1,j,i)+\delta_{k}) \tag{21}%
\end{equation}
for $i=1,2,...,p.$ Therefore, it follows from (6) that the gaps in the
spectrum are the subset of the $\varepsilon_{k}$ extension of the images
$\mu_{k,j}((c,d))$ of the intervals $(c,d)$ of (21). Since
\[
\mu_{k,j}(t(2k,j,i))=\left(  2\pi k\right)  ^{2}+\frac{\mu_{i}+\mu_{j}}%
{2}+\left(  \frac{\mu_{i}-\mu_{j}}{8\pi k}\right)  ^{2}%
\]
and
\[
\mu_{k,j}(t(2k+1,j,i))=\left(  2\pi k+\pi\right)  ^{2}+\frac{\mu_{i}+\mu_{j}%
}{2}+\left(  \frac{\mu_{i}-\mu_{j}}{4\pi(2k+1)}\right)  ^{2},
\]
using (6) and (16) we obtain that for any interval $(c,d)$ of (21) the
interval $\mu_{k,j}((c,d))$ and hence its $\varepsilon_{k}$ extension are
contained in $S(j,k)$ for each $j=1,2,...,p.$ The corollary is proved.
\end{proof}

Now we find a condition on the eigenvalues of the matrix $C$ for which the
spectrum of $L(Q)$ contains the interval $(H,\infty)$ for some constant $H$.
If the matrix $C$ has only one eigenvalue $\mu$ with multiplicity $m,$ then it
is possible that the spectrum of $L(Q)$ has infinitely many gaps. For example,
if $Q=qI,$ where $q$ is not a finite zone scalar potential and $I$ is the
$m\times m$ unit matrix then the spectrum of $L(Q)$ has infinitely many gaps.
If the matrix $C$ has only two eigenvalues $\mu_{1}$ and $\mu_{2},$ then the
sets $S(1,k)$ and $S(2,k)$ have a common interval
\[
\left(  \left(  2\pi k\right)  ^{2}+\frac{\mu_{1}+\mu_{2}}{2}-\gamma
_{k},\left(  2\pi k\right)  ^{2}+\frac{\mu_{1}+\mu_{2}}{2}+\gamma_{k}\right)
.
\]
Therefore, Corollary 2 does not imply that the number of the gaps in the
spectrum of $L(Q)$ is finite. However, we prove that if the number of
different eigenvalues of the matrix $C$ is greater than $2,$ then three sets
$S(j_{1},k),S(j_{2},k)$ and $S(j_{3},k)$ for the large values of $k$ have no
common intervals if the following condition holds.

\begin{condition}
Suppose that there exists a triple $(j_{1},j_{2},j_{3})$ such that
\[
\min_{i_{1},i_{2},i_{3}}\left(  diam(\{\mu_{j_{1}}+\mu_{i_{1}},\mu_{j_{2}}%
+\mu_{i_{2}},\mu_{j_{3}}+\mu_{i_{3}}\})\right)  =d\neq0,
\]
where minimum is taken under condition $i_{s}\in\left\{  1,2,...,p\right\}  $
for $s=1,2,3$ and
\[
diam(E)=\sup_{x,y\in E}\mid x-y\mid.
\]

\end{condition}

Let us first discuss why gaps in $\sigma(L(Q))$ do not appear in the interval
$(H,\infty)$ if $H$ is a large number and Condition 1 holds. Then in Theorem 4
we give the mathematical proof of this statement. For each $j\in\left\{
1,2,3\right\}  $ the set
\[
\sigma_{j}\left(  L(C)\right)  =\left\{  \left(  2\pi k+t\right)  ^{2}+\mu
_{j}:k\in\mathbb{Z},\text{ }t\in(-\pi,\pi]\right\}
\]
(let us call it $j$ spectrum) cover the interval $(H,\infty).$ The
perturbation $Q-C$ may generate a gap in $\sigma_{j}\left(  L(C)\right)  $
only at the neighborhood of the exceptional Bloch eigenvalues $\left(  2\pi
k+t\right)  ^{2}+\mu_{j}$ (let us call it $j$ exceptional Bloch eigenvalue).
On the other hand, Condition 1 implies that the $j_{1},$ $j_{2}$ and $j_{3}$
exceptional Bloch eigenvalues have no common points. That is why, for each
$\lambda\in(H,\infty)$ there exists $s\in\left\{  1,2,3\right\}  $ such that
$\lambda$ does not belong to the neighborhood of $j_{s}$\ exceptional Bloch
eigenvalues. Hence the perturbation $Q-C$ does not generate a gap in
$\sigma_{j_{s}}\left(  L(C)\right)  $ at the neighborhood of $\lambda.$

Now using Condition 1 and Corollary 2 we prove the following.

\begin{theorem}
If the matrix $C$ has three eigenvalues $\mu_{j_{1}},$ $\mu_{j_{2}}$ and
$\mu_{j_{3}}$ satisfying Condition 1, then there exists a number $H$ such that
$(H,\infty)\subset\sigma(L(Q)),$ that is, the number of the gaps in the
spectrum of $L(Q)$ is finite.
\end{theorem}

\begin{proof}
By Corollary 2 the gap $(\alpha,\beta)$ lying\ in $U(k)$ for $k>N_{3}$ belong
to the set $S(j_{s},k,)$ for all $s\in\left\{  1,2,3\right\}  $ and for some
$k>N_{3}.$ Therefore it is enough to prove that
\begin{equation}%
{\textstyle\bigcap\limits_{s=1,2,3}}
S(j_{s},k) \tag{22}%
\end{equation}
is an empty set for $k>N_{3}$. Since $\gamma_{k}\rightarrow0$ the number
$N_{3}$ can be chosen so that $4\gamma_{k}<d$ for $k>N_{3}$. If the set (22)
contains an element $x,$ then using the definitions of $S(j_{s},k)$, we obtain
that there exist $k>N_{3}$ and $i_{s}\in\left\{  1,2,...,p\right\}  $ such
that
\[
\mid x-(\pi k))^{2}-\frac{\mu_{j_{s}}+\mu_{i_{s}}}{2}\mid<\gamma_{k}%
\]
for all $s=1,2,3$. This implies that
\[
\left\vert \left(  \mu_{j_{u}}+\mu_{i_{u}}\right)  -\left(  \mu_{j_{v}}%
+\mu_{i_{v}}\right)  \right\vert <4\gamma_{k}<d
\]
for all $u,v\in\left\{  1,2,3\right\}  $, where $v\neq u.$ This contradicts
Condition 1.
\end{proof}

\end{document}